\documentclass[a4paper, intlimits, reqno]{amsart}

\usepackage[english]{babel}
\usepackage[T1]{fontenc}
\usepackage[utf8]{inputenc}

\usepackage{amsmath}
\usepackage{amssymb}
\usepackage{MnSymbol}
\usepackage{amsthm}
\allowdisplaybreaks
\usepackage{amsfonts}
\usepackage{mathrsfs} 
\usepackage{mathtools}
\usepackage{nicefrac}
\usepackage{enumitem}

\usepackage{multicol}
\usepackage{url}
\usepackage{dsfont}
\usepackage[numbers,sort&compress]{natbib}
\usepackage{doi}
\usepackage{prettyref}
\usepackage{xcolor}
\usepackage{tikz}

\newtheorem{lemma}{Lemma}[section]
\newtheorem{proposition}[lemma]{Proposition}
\newtheorem{theorem}[lemma]{Theorem}
\newtheorem{corollary}[lemma]{Corollary}
\newtheorem{example}[lemma]{Example}

\theoremstyle{definition}
\newtheorem{definition}[lemma]{Definition}
\newtheorem{remark}[lemma]{Remark}

\newcommand{\N}{\mathds{N}}

\newcommand{\R}{\mathds{R}}
\newcommand{\C}{\mathds{C}}
\newcommand{\K}{\mathds{K}}

\renewcommand{\P}{\mathds{P}}
\newcommand{\E}{\mathds{E}}
\newcommand{\X}{\mathds{X}}

\newcommand{\topo}{\tau}

\newcommand{\calB}{\mathcal{B}}
\newcommand{\calE}{\mathcal{E}}
\newcommand{\calF}{\mathcal{F}}

\providecommand{\abs}[1]{\left\lvert#1\right\rvert}
\providecommand{\norm}[1]{\left\lVert#1\right\rVert}
\providecommand{\differential}{\mathrm{d}}

\renewcommand{\d}{\differential}

\renewcommand{\exp}{\mathrm{e}}
\newcommand{\from}{\colon}

\newcommand\rlim{
\mathchoice{\vcenter{\hbox{${\scriptstyle{+}}$}}}
{\vcenter{\hbox{$\scriptstyle{+}$}}}
{\vcenter{\hbox{$\scriptscriptstyle{+}$}}}
{\vcenter{\hbox{$\scriptscriptstyle{+}$}}}}
\DeclareMathAlphabet{\mathcal}{OMS}{cmsy}{m}{n}
\newcommand{\vertiii}[1]{{\left\vert\kern-0.25ex\left\vert\kern-0.25ex\left\vert #1 
    \right\vert\kern-0.25ex\right\vert\kern-0.25ex\right\vert}}

\begin{document}

\title[Subordination]{Subordination for sequentially equicontinuous equibounded $C_0$-semigroups}
\author[K.~Kruse]{Karsten Kruse}
\address{Hamburg University of Technology\\ Institute of Mathematics \\ Am Schwarzenberg-Campus~3 \\
21073 Hamburg \\
Germany}
\email{karsten.kruse@tuhh.de}
\author[J.~Meichsner]{Jan Meichsner}
\email{jan.meichsner@tuhh.de}
\author[C.~Seifert]{Christian Seifert}
\email{christian.seifert@tuhh.de}



\date{\today}
\begin{abstract}
  We consider operators $A$ on a sequentially complete Hausdorff locally convex space $X$ such that $-A$ generates a (sequentially) equicontinuous equibounded $C_0$-semigroup.
  For every Bernstein function $f$ we show that $-f(A)$ generates a semigroup which is of the same `kind' as the one generated by $-A$.
  As a special case we obtain that fractional powers $-A^{\alpha}$, where $\alpha \in (0,1)$, 
  are generators.
\end{abstract}
\maketitle

\section{Introduction}
\label{sect:intro}

In this paper we aim to generalise subordination of bounded $C_0$-semigroups from the well-known case of Banach spaces to sequentially complete Hausdorff locally convex spaces.

The theory of $C_0$-semigroups on Banach spaces is by now a classical topic, see e.g.\ the monographs \cite{Davies1980,EngelNagel2000,Goldstein1985,HillePhillips1957,Pazy1983}.
It has been generalised to locally convex spaces in various contexts in \cite[Chapter IX]{yosida1968} and \cite{budde2019,Choe1985,Dembart1974,federico2016,Jefferies1986,Jefferies1987,Komura1968,Kuehnemund2001}. 
Since on these spaces continuity and sequential continuity may differ, we work with the (weaker) notion of sequential (equi-) continuity as in \cite{federico2016}.

Subordination (in the sense of Bochner) for bounded $C_0$-semigroups on Banach spaces describes a technique to associate a new semigroup to a given one by integrating orbits against a convolution semigroup of measures.
It plays an important role in operator theory, functional calculus theory and stochastic processes, see e.g.\ \cite{bochner1949,phillips1952,Schilling1998}.
As is well-known (see also Proposition \ref{prop:Laplace-Transform}) these convolution semigroups of measures correspond via Laplace transform to the class of Bernstein functions, cf.\ the monograph \cite{schilling2012}.
It turns out that the generator of the subordinated semigroup can be described by means of the Bernstein function and the generator of the original semigroup \cite[Theorem 4.3]{phillips1952}.

Although the framework of $C_0$-semigroups on Banach spaces yields a rich theory as described above, for example even the classical heat semigroup on $C_{\rm b}(\R^n)$ does not fit in this context, 
however can be treated in our generalised setting; cf.\ Example \ref{ex:heat_semigroup} below.

Let us outline the content of the paper.
In Section \ref{chap:integration} we review the Pettis-integral which provides a suitable integral in our context of locally convex spaces, 
in particular those satifying the so-called metric convex compactness property, see Theorem \ref{thm:metric_convex_compactness_property} below.
The theory of locally sequentially equicontinuous, equibounded $C_0$-semigroups on sequentially complete Hausdorff locally convex spaces can be developped analogously to the 
classical theory of $C_0$-semigroups on Banach spaces, apart from the fact that the continuity properties need to be described by sequences.
We collect the facts needed in Section \ref{chap:equicontinoussgs}.
We then turn to subordination in this context in Section \ref{chap:subordination}.
After introducing Bernstein functions, the right class of functions for this purpose, we develop the theory of subordination in our generalised setting.
It turns out that the subordinated semigroup is again a locally sequentially equicontinuous, equibounded $C_0$-semigroup (see Proposition \ref{prop:subordinatedsg}) and 
that its generator can be related to the one of the original semigroup (see Theorem \ref{thm:representationonDA} and Corollary \ref{coro:generator_subordinated_semigroup}).
These are our main abstract results. We will then apply these results to bounded (locally) bi-continuous semigroups (as introduced in \cite{kuehnemund2003}) and to transition semigroups for Markov processes in Section \ref{chap:bi-continuous}.
In particular, the above-mentioned classical heat semigroup, also called Gau{\ss}-Weierstra{\ss} semigroup, on $C_{\rm b}(\R^n)$ fits into this context.

\section{Integration in locally convex spaces}
\label{chap:integration}

In this section we review the notion of the Pettis-integral.

\begin{definition}[Pettis-integral]
\label{def:integral}
	Let $(\Omega, \Sigma, \mu)$ be a measure space and $X$ be a Hausdorff locally convex space. 
	A function $f: \Omega \rightarrow X$ is called \emph{weakly (scalarly) essentially measurable} if the function $\langle x' , f \rangle: \Omega \rightarrow \K$, $\langle x' , f \rangle (\omega) := \langle x' , f(\omega) \rangle$ is essentially measurable (i.e.\ it has a measurable representative) for all $x' \in X'$. 
	Here $X'$ denotes the topological dual space of $X$ and $\langle \cdot , \cdot \rangle$ is the canonical pairing. 
	A weakly essentially measurable function is said to be \emph{weakly (scalarly) integrable} if $\langle x' , f \rangle \in L^1 (\Omega, \mu)$. 
	A weakly integrable function $f$ is called \emph{($\mu$-Pettis-) integrable} if
	\[
	 \exists x \in X \, \forall x' \in X': \langle x' , x \rangle = \int\limits_{\Omega} \langle x' , f(\omega) \rangle \mu(\d \omega). 
	\]  
	In this case $x$ is unique due to the Hausdorff property and we set 
	\[
	 \int\limits_{\Omega} f(\omega) \mu(\d \omega):=x.
	\]
	
\end{definition}

Recall that a Hausdorff locally convex space $X$ is said to have the \emph{metric convex compactness property} if for each metrisable compact subset $C\subseteq X$ also the closed convex hull $\overline{\operatorname{cx}} \, C$ of $C$ is compact. 
Note that if $X$ is sequentially complete and Hausdorff locally convex, then $X$ has the metric convex compactness property by \cite[Remark 4.1.b]{voigt1992}.

\begin{theorem}[{\cite[Theorem 0.1]{voigt1992}}]
\label{thm:metric_convex_compactness_property}
Let $X$ be a Hausdorff locally convex space. Then the following are equivalent.
\begin{enumerate}
	\item $X$ has the metric convex compactness property.
	\item If $\Omega$ is a compact metric space, $\mu$ a finite Borel measure on $\Omega$ and $f: \Omega \rightarrow X$ continuous, then $f$ is $\mu$-Pettis-integrable. 
\end{enumerate}	
\end{theorem}

If $X$ is a Hausdorff locally convex space with metric convex compactness property and $\mu$ a positive and finite Borel measure on $[0,\infty)$,
then by Theorem \ref{thm:metric_convex_compactness_property} every continuous function $f: [0, \infty) \rightarrow X$ is $\mu$-Pettis-integrable over every compact interval $[a,b] \subseteq [0, \infty)$ with $0\leq a<b$. 
In this case we can construct the Pettis-integral explicitly by using Riemann sums.
To that end, fix $0\leq a< b$ and for $n\in\N$ let $( a = x_0^{(n)}, \ldots , x_{k_n}^{(n)} = b)$ be a partition of $[a,b]$ such that $\max_k (x_k^{(n)}-x_{k-1}^{(n)}) \to 0$,
and $\{ \xi^{(n)}_i \mid \xi^{(n)}_i \in (x^{(n)}_{i-1}, x^{(n)}_i],\, i \in \{1, \dots, k_n \}\}$ a corresponding set of intermediate points. Then
\[
 \int\limits_{[a,b]} f (\lambda) \mu ( \d \lambda ) = \lim\limits_{n \rightarrow \infty} \sum\limits_{i=1}^{k_n} f \left( \xi^{(n)}_i \right) \mu \bigl( (x^{(n)}_{i-1}, x^{(n)}_i] \bigr) + f(a) \mu \bigl( \{a\} \bigr). 
\]
If $f$ is additionally bounded, by boundedness of $\mu$ we can also integrate $f$ with respect to $\mu$ over the entire interval $[0,\infty)$.

\begin{lemma}
\label{lemma:integrability}
  Let $X$ be a sequentially complete Hausdorff locally convex space, $\mu$ a finite Borel measure on $[0, \infty)$ and $f \in C_{\rm b}([0, \infty),X)$.
  Then $f$ is $\mu$-Pettis-integrable.
\end{lemma}

\begin{proof}
	By the above we can integrate $f$ over every compact interval. 
	Consider the sequence $(x_n)$ defined by
	\[
	 x_n := \int\limits_{[0, n]} f(\lambda) \mu (\d \lambda) \quad (n\in \N).
	\]
	Let $n,m \in \N$, $m<n$. For every seminorm $\norm{\cdot}_p$ of the family of seminorms generating the topology of $X$ we have
	\[
	 \norm{x_n - x_m}_p \leq \norm{f}_{p,\infty} \mu \bigl( (m, n] \bigr),
	\]
	implying that $(x_n)$ is a Cauchy sequence.
	Since $X$ is sequentially complete, $(x_n)$ is convergent. Let $x:=\lim_{n\to\infty} x_n$.
	Then the equation $x = \int_{[0, \infty)} f(\lambda) \mu (\d \lambda)$ is verified by a direct calculation.  
\end{proof}

\section{Locally Sequentially Equicontinuous Semigroups}
\label{chap:equicontinoussgs}

Let $(X,\topo)$ be a Hausdorff locally convex space. 
The system of seminorms generating the topology $\topo$ will be denoted by $(\norm{\cdot}_p)_{p \in \mathcal{P}}$. W.l.o.g.\ we may assume that $(\norm{\cdot}_p)$ is directed. 		

\begin{definition}
  \label{defi:equicontinoussg}
  A \emph{semigroup} $(T_t)_{t\geq 0}$ of linear operators defined on $X$ is a family of linear operators $T_t\from X\to X$ for $t\geq 0$, such that $T_0 = I$ and $T_{s+t} = T_s T_t$ for all $s,t\geq 0$.
  A semigroup $(T_t)_{t \geq 0}$ on $X$ is said to be
  \begin{enumerate}[itemsep=4mm]
    \item
      a \emph{$C_0$-semigroup on $(X,\tau)$} if for all $x\in X$ we have $\lim\limits_{t \to 0\rlim} T_t x = x$,
    \item
    \label{prop:equicontinoussg1}
      \emph{locally sequentially equicontinuous} if for all sequences $(x_n)$ in $X$ such that $x_n \rightarrow 0$, $t_0 > 0$ and $p \in \mathcal{P}$ we have 
      $\lim\limits_{n \rightarrow \infty} \sup\limits_{t \in [0, t_0]} \norm{T_t x_n}_p = 0$,
    \item
    \label{prop:equicontinoussg2}
      \emph{locally equibounded} if for all bounded sets $B \subseteq X$, $t_0 > 0$ and $p\in\mathcal{P}$ it holds that  $\sup\limits_{\substack{x \in B \\  t \in [0,t_0]}} \norm{T_t x}_p < \infty$.
  \end{enumerate}	
  We drop the word `locally' if the properties (b) and (c) hold uniformly on $[0, \infty)$. 
\end{definition}

\begin{remark}
  If $(T_t)$ is a (locally) sequentially equicontinuous semigroup, then $(T_t)$ is (locally) equibounded (\cite[Propositions A.1 (iii)]{federico2016}).
  Moreover, if $(T_t)$ is a locally sequentially equicontinuous $C_0$-semigroup on $(X, \tau)$, then the mapping $[0,\infty)\ni t\mapsto T_tx\in X$ is continuous for all $x\in X$.
\end{remark}

Let $(T_t)$ be a locally sequentially equicontinuous, equibounded $C_0$-semigroup on $(X,\topo)$.  
As in the case of $C_0$-semigroups on Banach spaces we define the \emph{generator} $-A$ of $(T_t)$ by
\begin{align*}
  \mathcal{D}(A)& := \bigl\{x\in X \mid \lim_{t\to 0\rlim} \tfrac{1}{t} (T_tx-x) \,\text{exists} \bigr\},\\
  -Ax & := \lim_{t\to 0\rlim} \tfrac{1}{t} (T_tx-x) \quad(x\in \mathcal{D}(A)).
\end{align*}

\begin{remark}
 There is no common agreement whether to use the here presented definition of a generator or its negative. Throughout the entire paper we will stick to the above made definition, i.e.\ $-A$ is the generator. 
\end{remark}

Let $\rho_0(-A)\subseteq \C$ be the set of all elements $\lambda\in \C$ such that the operator $\lambda + A $ has a sequentially continuous inverse (note that we allow for complex values here). 

Let us collect some basic facts for $(T_t)$ and $A$.

\begin{lemma}[see {\cite[Propositions 3.10, 3.11, Theorem 3.14, Corollary 3.15 and Corollary 3.16]{federico2016}}]
\label{lem:properties_generator}
  Let $X$ be a sequentially complete Hausdorff locally convex space, $(T_t)$ a locally sequentially equicontinuous, equibounded $C_0$-semigroup on $(X,\topo)$ with generator $-A$.
  Then
  \begin{enumerate}
    \item
      $\mathcal{D}(A)$ is sequentially dense in $X$,
    \item
     for $x\in \mathcal{D}(A)$ we have $T_t x\in \mathcal{D}(A)$ for all $t \geq 0$, $t\mapsto T_t x$ is differentiable and $\frac{\d}{\d t} T_tx = -AT_tx = -T_tAx$ for all 
     $t \geq 0$,    
    \item
     $(0,\infty) \subseteq  \rho_0(-A)$ and
     \[
	 (\lambda + A)^{-1}x = \int\limits_{(0, \infty)} \exp^{-\lambda t} T_t x \, \d t \quad (x\in X, \lambda > 0),
     \]    
    \item
     $A$ is sequentially closed,
    \item
     for all $x \in X$ one has $\lim\limits_{\lambda \rightarrow \infty} \lambda (\lambda + A)^{-1}x = x$. 
  \end{enumerate}
\end{lemma}

\begin{remark} \leavevmode
 \begin{enumerate}
  \item
   Note that compared to \cite{federico2016}, we only assume local sequential equicontinuity for $(T_t)$. However, this does not affect the results and proofs.
  \item
   Unless we assume $(T_t)$ to be (sequentially) equicontinuous on the whole of $[0, \infty)$, we cannot show that the resolvent family ${\bigl( \lambda (\lambda + A )^{-1} \bigr)_{\lambda > 0}}$ corresponding to 
   $(T_t)$ is (sequentially) equicontinuous.
   However, for $\varepsilon > 0$ the rescaled resolvent family $\bigl( \lambda (\lambda + A + \varepsilon)^{-1} \bigr)_{\lambda>0}$ corresponding to the (sequentially) equicontinuous semigroup 
   $(\exp^{-\varepsilon t} T_t)$ is (sequentially) equicontinuous. 
 \end{enumerate}
\end{remark}	

\section{Subordination for equicontinuous $C_0$-Semigroups}
\label{chap:subordination}

We start with the definition of the class of functions in which we will plug in the negative of a generator of a locally sequentially equicontinuous, equibounded $C_0$-semigroup in order to get a new generator. 

\begin{definition}[Bernstein function]
	Let $f\from(0,\infty)\to[0,\infty)$. 
	Then $f$ is called \emph{Bernstein function} provided $f \in C^\infty(0,\infty)$ and $(-1)^{k-1} f^{(k)} \geq 0$ for all $k\in\N$.
\end{definition}

Bernstein functions appear in a vast number of fields  such as probability theory, harmonic analysis, complex analysis and operator theory under different names, e.g.\ Laplace exponents, negative definite functions or Pick, Nevanlinna or Herglotz functions (complete Bernstein functions, cf.\ \cite{schilling2012}).  
They allow for a very useful representation formula in terms of measures.  

\begin{proposition}[{\cite[Theorem 3.2]{schilling2012}}]
  \label{prop:levy_khintchine}
  Let $f\from(0,\infty)\to [0,\infty)$. The following are equivalent.
  \begin{enumerate}
    \item
      $f$ is a Bernstein function.
    \item
      There exist constants $a,b\geq 0$ and a positive Radon measure $\mu$ on $(0,\infty)$ satisfying $\int_{(0,\infty)} 1\wedge t\,\mu(\d t) <\infty$ such that
      \begin{equation}
	f(\lambda) = a + b \lambda + \int\limits_{(0,\infty)} \bigl(1-e^{-\lambda t}\bigr)\,\mu(\d t) \quad (\lambda > 0) \label{eq:representation}.
      \end{equation}
  \end{enumerate}
\end{proposition}
The representation of $f$ in \eqref{eq:representation} is called \emph{L\'{e}vy-Khintchine representation}. The function $f$ determines the two numbers $a$, $b$ and the measure $\mu$ uniquely. The triplet $(a,b,\mu)$ is called \emph{L\'{e}vy triplet} of $f$. 

Every Bernstein function admits a continuous extension to $[0, \infty)$ since by applying dominated convergence to the representation formula \eqref{eq:representation} one gets $f(0\rlim) = a$. 

\begin{example}
  Let $\alpha\in (0,1)$ and $f\from (0,\infty)\to [0,\infty)$ be defined by $f(x):=x^\alpha$ for all $x>0$.
  Then $f$ is a Bernstein function with L\'{e}vy triplet $(0,0,\mu)$, where for measurable sets $B\subseteq (0,\infty)$ we have
  \[\mu(B) := \frac{-1}{\Gamma (-\alpha)} \int\limits_B t^{-1-\alpha}\, \d t.\]
  Hence, 
  \[
  	x^{\alpha} = \frac{1}{\Gamma (-\alpha)} \int\limits_{(0, \infty)} \bigl( e^{-xt} - 1 \bigr) t^{-1-\alpha} \d t. 
  \]
\end{example}

Let us now turn to a concept closely related to Bernstein functions. 

\begin{definition}
\label{def:convolutionsemigroup}
  Let $(\mu_t)_{t \geq 0}$ be a family of Radon measures on $[0,\infty)$, $\mu$ a Radon measure on $[0,\infty)$. Then $(\mu_t)$ is called 
  \begin{enumerate}
    \item
      a family of \emph{sub-probability measures} if $\forall t \in [0, \infty): \, \mu_t \big( [0, \infty) \big) \leq 1$,
    \item
      \emph{convolution semigroup} if $\mu_0 =\delta_0$ and $\forall s,t \in [0, \infty): \, \mu_t \ast \mu_s = \mu_{t+s}$,  
    \item
      \label{def:convolutionsemigroup:apoint}
      \emph{vaguely continuous at $s \in [0, \infty)$} with limit $\mu$ if 
      \[
	\forall f \in C_{\rm c}[0, \infty): \, \lim\limits_{t \rightarrow s }\int\limits_{[0, \infty)} f(\lambda) \,\mu_t (\d \lambda) =  
	\int\limits_{[0, \infty)} f(\lambda) \,\mu (\d \lambda). 
      \] 
  \end{enumerate}	 	
\end{definition}	

\begin{remark} \leavevmode
\label{remark:vaguely-weak}
  \begin{enumerate} 
    \item\label{remark:vaguely-weak:item:1}
      A family $(\mu_t)$ of sub-probability measures which is vaguely continuous at $0$ with limit $\delta_0$ is also \emph{weakly continuous at $0$}, i.e.~(c) in 
      Definition \ref{def:convolutionsemigroup} actually holds for all $f \in C_{\rm b}[0, \infty)$.
      In order to see this, choose $f \in C_{\rm c}[0, \infty)$ such that $0\leq f \leq 1$ and $f=1$ in a neighbourhood of $0$. Then
      \[
      1 \geq \mu_t \big( [0, \infty) \big) \geq \int\limits_{[0, \infty)} f(\lambda) \, \mu_t (\d \lambda)  \rightarrow f(0) = 1 \text{ as } t \rightarrow 0\rlim, 
      \] 
      i.e.\ $\mu_t\bigl([0,\infty)\bigr)\to 1$.
      By \cite[Theorem A.4]{schilling2012} this implies weak continuity at $0$.
    \item
      Let $(\mu_t)$ be a convolution semigroup of sub-probability measures which is vaguely continuous at $0$ with limit $\delta_0$. Then $(\mu_t)$ is vaguely continuous at every point $s \geq 0$ with limit $\mu_s$.
      Indeed, we can define a contractive semigroup via
      \[
      	(T_t f)(x) := \int\limits_{[0, \infty)} f(x+\lambda) \,\mu_t(\d \lambda) \quad(x\in [0,\infty), f\in C_0[0,\infty))
      \]
      on the Banach space $C_0 [0, \infty) = \overline{C_{\rm c} [0, \infty)}$. 
      We claim that $(T_t)$ is strongly continuous. 
      Then vague continuity of $(\mu_t)$ follows since this also implies weak continuity and $\delta_0\in C_0[0,\infty)'$.
      To show the claim, let $f\in C_0[0,\infty)$. For $\varepsilon>0$, by uniform continuity of $f$, there is $\delta>0$ such that $\abs{f(x+\lambda)-f(x)}<\varepsilon$ if $\lambda\in [0,\delta)$ and $x\in [0,\infty)$. 
      By (a), one actually sees $\mu_t([0,\delta))\to 1$ and consequently $\mu_t([\delta,\infty))\to 0$ as $t\to 0\rlim$. Thus,
      \begin{align*}
        & \abs{T_tf(x)-f(x)}= \Bigl{|}\int\limits_{[0,\infty)} f(x+\lambda)\,\mu_t(\d \lambda) - f(x)\Bigr{|} \\
        & \leq \int\limits_{[0,\delta)} \abs{ f(x+\lambda) - f(x)}\,\mu_t(\d \lambda) + 2\norm{f}_\infty \mu_t([\delta,\infty)) + \norm{f}_{\infty} (1-\mu_t([0,\infty)))\\
        & \leq \varepsilon + 2\norm{f}_\infty \mu_t([\delta,\infty)) + \norm{f}_{\infty} (1-\mu_t([0,\infty))), 
      \end{align*}
      which can be made arbitrarily small, uniformly in $x$. 
      The estimate shows strong continuity at $t=0$ and by standard arguments this holds for all $t \geq 0$. 
  \end{enumerate}
\end{remark}

Every Bernstein function is naturally associated to a family $(\mu_t)$ of sub-proba\-bility measures which form a vaguely (and hence weakly) continuous convolution semigroup and vice versa. 

\begin{proposition}[{\cite[Theorem 5.2]{schilling2012}}]
\label{prop:Laplace-Transform}
  Let $(\mu_t)$ be a convolution semigroup of sub-proba\-bility measures on $[0, \infty)$ 
   which is vaguely continuous at $0$ with limit $\delta_0$. 
  Then there exists a unique Bernstein function $f$ such that for all $t\geq 0$ the Laplace transform of $\mu_t$ is given by 
  \[
    \mathcal{L} (\mu_t) = \exp^{-tf}. 
  \]
  Conversely, given any Bernstein function $f$, there exists a unique vaguely continuous convolution semigroup $(\mu_t)$ of sub-probability measures on $[0, \infty)$ such that the above equation holds. 
\end{proposition}

By the above proposition we obtain that the sub-probability measures $\mu_t$ are probability measures if and only if $f(0\rlim) = 0$, since
\begin{equation}
  \mu_t \big( [0, \infty) \big) = \lim\limits_{\lambda \rightarrow 0\rlim} \int\limits_{[0, \infty)} \exp^{-\lambda s} \mu_t (\d s) = \lim\limits_{\lambda \rightarrow 0\rlim} \exp^{-tf(\lambda)} = \exp^{-tf(0\rlim)}. \label{eq:subprobabilitymeasures}
\end{equation}

\begin{example}
  Let $\alpha\in (0,1)$ and $f\from (0,\infty)\to [0,\infty)$ be defined by $f(x):=x^\alpha$ for all $x>0$.
  Then for $t > 0$ the measure $\mu_t$ has a density $g_t$ w.r.t.\ the Lebesgue measure given by
  \[
  	g_t (s) =  \frac{1}{2 \pi i} \int\limits_{\gamma_{\Theta}} \exp^{-tw^{\alpha}} \exp^{s w} \d w \quad (s > 0), 
  \]
  where $\gamma_{\Theta} = \gamma_{\Theta}^+ \cup \gamma_{\Theta}^-$ is parametrised by
  \[
  	\gamma_{\Theta}^{-}(r) := -r \exp^{-i \Theta} \quad (r\in (-\infty,0)),\qquad \gamma_{\Theta}^{+}(r) := r \exp^{i \Theta} \quad (r\in (0,\infty))
  \]
  and $\Theta\in[\pi /2 , \pi]$.
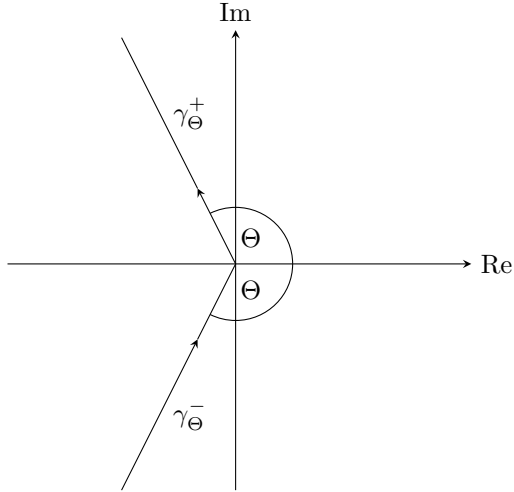
  \begin{figure}[htb]
    \centering
    \begin{tikzpicture}[>=stealth,scale=0.75]
      \draw[->] (-3,0)--(3.1,0) node[right]{$\mathrm{Re}$};
      \draw[->] (0,-3)--(0,3.1) node[above]{$\mathrm{Im}$};
      \draw[->] (-1.5,-3)--(-0.5,-1);
      \draw[->] (-0.5,-1)--(0,0)--(-0.5,1);
      \draw (-0.5,1)--(-1.5,3);
      \draw (0.75,0) arc(0:116.565051177078:0.75);
      \draw (0.75,0) arc(0:-116.565051177078:0.75);
      \draw (0.2,0.35) node{$\Theta$};
      \draw (0.2,-0.35) node{$\Theta$};
      \draw (-0.6,-2) node{$\gamma_{\Theta}^-$};
      \draw (-0.6,2) node{$\gamma_{\Theta}^+$};
    \end{tikzpicture}  
  \end{figure}
  
  For $\alpha=\tfrac{1}{2}$ one can explicitly calculate the integral and finds (see \cite[p. 259-268]{yosida1968} for details)
  \[
  	g_t (s) = \frac{t \exp^{-t^2/(4s)}}{2\sqrt{\pi}s^{3/2}}  \quad (s > 0).
  \] 
\end{example}

Recall that a family $(\mu_t)_{t\in I}$ of sub-probability measures on $[0,\infty)$, where $I \subseteq [0, \infty)$, is called \emph{uniformly tight} if for all $\varepsilon>0$ there exists $K\subseteq [0,\infty)$ compact such that
\[\sup_{t\in I} \mu_t\bigl([0,\infty)\setminus K\bigr)\leq \varepsilon.\]

\begin{lemma}
\label{lemma:tightness}
  Let $(\mu_t)_{t \geq 0}$ be a weakly continuous family of sub-probability measures on $[0,\infty)$ and $J \subseteq [0, \infty)$ be compact. 
  Then the sub-family $(\mu_t)_{t \in J}$ is uniformly tight. 
\end{lemma}

\begin{proof}
  This is just a direct application of Prohorov's theorem \cite[Theorem 8.6.2]{bogachev2006} for which we need to show the existence of a weakly convergent subsequence in a given sequence $( \mu_{t_n})$.
  But this follows from the fact that the mapping $t \mapsto \mu_t$ is continuous with respect to the weak topology of measures and the compactness of $J$.   
\end{proof}

Analogously to the case of bounded $C_0$-semigroups on Banach spaces (see \cite[Proposition 13.1]{schilling2012}) we can construct a new (locally) sequentially equicontinuous, equibounded $C_0$-semi\-group from an existing one using a vaguely continuous convolution semigroup $(\mu_t)$ of sub-probability measures.

\begin{proposition}
\label{prop:subordinatedsg}
  Let $X$ be a sequentially complete Hausdorff locally convex space, $(T_t)$ be a (locally) sequentially equicontinuous, equibounded $C_0$-semigroup on $(X,\topo)$ and $(\mu_t)$ be a 
  convolution semigroup of sub-pro\-ba\-bi\-li\-ty measures which is vaguely continuous at $0$ with limit $\delta_0$.
  For $t\geq 0$ define $S_t\from X\to X$ by
  \begin{equation}
  \label{eq:subordinatedsg}
    S_t x := \int\limits_{[0, \infty)} T_s x \, \mu_t (\d s) \quad(x\in X).
  \end{equation}
  Then $(S_t)$ is a (locally) sequentially equicontinuous, equibounded $C_0$-semigroup on $(X,\topo)$. 
\end{proposition}

We will call $(S_t)$ the \emph{subordinated semigroup to $(T_t)$ w.r.t.\ $f$}, where $f$ is the Bernstein function associated to $(\mu_t)$. 

\begin{proof}
  Let $t\geq 0$ and $(\norm{\cdot}_p)_{p \in \mathcal{P}}$ the family of seminorms generating the topology $\topo$ of $X$. 
  By Lemma \ref{lemma:integrability} and equiboundedness and strong continuity of $(T_t)$ the above integral exists. 
  The linearity of $S_t$ is clear. 
  To show equiboundedness of $(S_t)$, for a bounded set $B \subseteq X$ and a seminorm $\norm{\cdot}_p$ we observe
  \[
    \sup\limits_{\substack{x \in B \\ t \in [0, \infty) }} \norm{S_t x}_p \leq \sup\limits_{\substack{x \in B \\ s \in [0, \infty) }} \norm{T_s x}_p.   
  \]
  Since $(T_t)$ is equibounded, so is $(S_t)$.
  The semigroup property of $(S_t)$ is inherited from the semigroup property of $(\mu_t)$ and of $(T_t)$. Indeed, let $s,t \geq 0$. For $x \in X$, $x' \in X'$ we have
  \begin{align*}
    \langle x' , S_t S_s x \rangle 
    & = \int\limits_{[0, \infty)}\quad\mathclap{\int\limits_{[0, \infty)}}\;\;\; \langle x', T_u T_v x \rangle \mu_s(\d u) \mu_t(\d v)
    = \int\limits_{[0, \infty)} \quad\mathclap{\int\limits_{[0, \infty)}}\;\;\; \langle x', T_{u+v} x \rangle \mu_s(\d u) \mu_t(\d v) \\
    & = \int\limits_{[0, \infty)}  \langle x', T_{w} x \rangle (\mu_s \ast \mu_t)(\d w)
    = \int\limits_{[0, \infty)}  \langle x' , T_{w} x \rangle \mu_{s + t}(\d w)
    = \langle x' , S_{t+s} x \rangle.
  \end{align*}
  
  For the strong continuity of $(S_t)$ let $x\in X$. For $p \in \mathcal{P}$ we estimate
  \[
    \norm{S_t x - x}_p \leq \int\limits_{[0, \infty)} \norm{T_s x - x}_p \mu_t (\d s) + \Bigl( 1 - \mu_t \bigl([0, \infty) \bigr) \Bigr) \norm{x}_p \to 0,
  \]
  since $[s\mapsto\norm{T_s x - x}_p]\in C_{\rm b}[0,\infty)$ with value $0$ at $0$ and $\mu_t\to \delta_0$ weakly.
  
  It remains to show that $(S_t)$ is (locally) sequentially equicontinuous.
  Let $(x_n)$ in $X$ be such that $x_n \to 0$.
  Let $t_0>0$, $p \in \mathcal{P}$ and $\varepsilon>0$.
  By Remark \ref{remark:vaguely-weak} $(\mu_t)$ is weakly continuous. Hence, by Lemma \ref{lemma:tightness} and equiboundedness of $(T_t)$ we can choose $s_0\geq 0$ such that
  \[
    \sup_{t\in [0,t_0]} \mu_t \bigl([s_0,\infty)\bigr) \cdot \sup\limits_{\substack{n \in \N \\ s \in [0, \infty)}} \norm{T_s x_n}_p \leq \frac{\varepsilon}{2}.
  \]
  By virtue of the local sequential equicontinuity of $(T_s)$ there exists $n_0 \in \N$ such that
  \[
    \sup\limits_{s \in [0, s_0]} \norm{T_s x_n}_p \leq \frac{\varepsilon}{2}
  \] 
  for all $n \geq n_0$. 
  Hence, for $n \geq n_0$ we obtain
  \begin{align*}
    \sup_{t\in [0,t_0]} \norm{S_t x_n}_p & \leq \sup_{t\in [0,t_0]} \int\limits_{[0,\infty)} \norm{T_sx_n}_p\, \mu_t(\d s)\\
    & \leq \sup_{t\in [0,t_0]} \int\limits_{[0, s_0]} \norm{T_s x_n}_p \, \mu_t (\d s) + \sup_{t\in [0,t_0]} \int\limits_{(s_0,\infty)} \norm{T_s x_n}_p \,\mu_t(\d s)\\
   & \leq \frac{\varepsilon}{2} + \frac{\varepsilon}{2} = \varepsilon.
  \end{align*}
  In case $(T_s)$ is even equicontinuous we may directly choose $n_0\in\N$ such that 
  \[
   \sup\limits_{s \in [0, \infty)} \norm{T_s x_n}_p \leq \varepsilon,
  \]  
  holds for $n \geq n_0$. 
  Hence, it follows that
  \[
   \sup\limits_{t \in [0, \infty)} \norm{S_t x_n}_p \leq \sup\limits_{t \in [0, \infty)} \int\limits_{[0,\infty)} \norm{T_sx_n}_p\, \mu_t(\d s) \leq \varepsilon ,
  \]
  which finishes the proof.
\end{proof}

\begin{definition}
  Let $X$ be a sequentially complete Hausdorff locally convex space, $(T_t)$ a (locally) sequentially equicontinuous, equibounded $C_0$-semigroup on $(X, \topo)$ with generator $-A$
  and $f$ a Bernstein function. 
  Then we will denote the generator of the subordinated semigroup $(S_t)$ by $-A^f$. 
\end{definition}

Our next goal is to represent the generator $-A^f$ of a subordinated semigroup $(S_t)$ for a given Bernstein function $f$ as it was performed in \cite[Eq.\ (13.10)]{schilling2012} for Banach spaces. We need some preparation. 
To begin with, we need to show that the function $s \mapsto T_s x-x$ can be approximated linearly in a neighbourhood of $s=0$ and thus is capable to compensate the measure appearing in the L\'{e}vy triplet $(a,b, \mu)$ which is singular at $s=0$. 

\begin{proposition}
  Let $X$ be a sequentially complete Hausdorff locally convex space, $(T_t)$ a (locally) sequentially equicontinuous, equibounded $C_0$-semigroup on $(X,\topo)$ with generator $-A$, 
  $f$ a Bernstein function with L\'{e}vy triplet $(a,b, \mu)$, and $x \in \mathcal{D}(A)$. 
  Then 
  \[
    (0, \infty) \ni s \mapsto T_s x-x \in X
  \]
  is $\mu$-Pettis-integrable.
  \label{prop:pettis_integrability_of_Ttx-x}
\end{proposition}

\begin{proof}
  Since $x \in \mathcal{D}(A)$, the mapping $s \mapsto T_s x$ is differentiable and 
  \[
    \frac{\d}{\d s} T_s x = - A T_s x = - T_s Ax \quad(s>0), 
  \]
  see Lemma \ref{lem:properties_generator}.
  Hence,
  \[
    T_t x-x = - \int\limits_{(0, t)} AT_s x \,\d s \quad (t\geq 0)
  \]
  and thus for every $x' \in X'$ there is a continuous seminorm $\norm{\cdot}_p$ (remember we assumed the family of seminorms to be directed) and a constant $C\geq 0$ such that
  \[
    \abs{\langle x' , T_t x-x \rangle} \leq C \norm{T_t x -x}_p \leq C \Bigl( \sup\limits_{s \in [0, \infty)} \norm{T_s Ax}_p \cdot t  \Bigr)\wedge \Bigl(\sup\limits_{s \in [0, \infty)} \norm{T_s x}_p + \norm{x}_p\Bigr)
  \]
  for all $t\geq 0$. Hence, $s \mapsto T_s x-x$ is $\mu$-weakly integrable by Proposition \ref{prop:levy_khintchine}.
  Further, by Theorem \ref{thm:metric_convex_compactness_property} we know that the function is integrable over every compact subset $K \subseteq (0, \infty)$. 
  Now we adapt the argument of the proof of Lemma \ref{lemma:integrability} by considering the sequences $(y_n)$ and $(z_n)$ defined by
  \[
    y_n := \int\limits_{[ \frac{1}{n}, r ]} ( T_s x-x ) \,\mu (\d s), \qquad
    z_n := \int\limits_{[ r , n ]} (T_s x-x) \,\mu (\d s),
  \]
  where $r > 0$ is a point of continuity of the measure $\mu$, i.e.\ $\mu(\{ r \}) = 0$. 
  As before one shows that both $(y_n)$ and $(z_n)$ are Cauchy sequences in $X$. Let their limits be denoted by $y$ and $z$, respectively. 
  Since
  \[y_n+z_n = \int\limits_{[\frac{1}{n},n]} (T_s x-x)\,\mu(\d s) \quad(n\in\N),\]
  we thus obtain
  \[
  z+y = \int\limits_{( 0, \infty )} (T_s x-x) \,\mu (\d s),
  \]
  which finishes the proof.
\end{proof}

\begin{lemma}
\label{lemma:invarianceofDA}
  Let $X$ be a sequentially complete Hausdorff locally convex space, $(T_t)$ a locally sequentially equicontinuous, equibounded $C_0$-semigroup on $(X,\topo)$ with generator $-A$,
  $f$ a Bernstein function, and $(S_t)$ the subordinated semigroup to $(T_t)$ w.r.t.\ $f$. 
  Then both $(S_t)$ and $\bigl(\lambda(\lambda + A^f)^{-1}\bigr)_{\lambda > 0}$ leave $\mathcal{D}(A)$ invariant and the operators of both families commute with $A$ (on $\mathcal{D}(A)$) and with the operators of $(T_t)$. 
\end{lemma}	

\begin{proof}
  The argument is the same as for the usual Banach space case. 
  Let $(\mu_t)$ be the family of measures associated to $f$ according to Proposition \ref{prop:Laplace-Transform}.
  Let $x \in \mathcal{D}(A)$. For $t\geq 0$ we have
  \[
    S_t x = \int\limits_{[0, \infty)} T_s x \, \mu_t(\d s) = \lim_{n\to\infty} \int\limits_{[0,n]} T_s x\,\mu_t(\d s)
  \]	
  by Lemma \ref{lemma:integrability}. For $n\in\N$ we can approximate $\int_{[0,n]} T_s x\,\mu_t(\d s)$ by (finite) Riemann sums which belong to $\mathcal{D}(A)$.
  Since $A$ is sequentially closed, also $\int_{[0,n]} T_s x\,\mu_t(\d s) \in \mathcal{D}(A)$ and therefore $S_tx\in \mathcal{D}(A)$.
  
  Let $s,t\geq 0$. We now show that $S_t$ commutes with $T_s$. Then $S_t$ also commutes with $A$ on $\mathcal{D}(A)$ since $S_t$ is sequentially continuous by Proposition \ref{prop:subordinatedsg}.
  Let $x\in X$. By Lemma \ref{lemma:integrability} and sequential continuity of $T_s$ we obtain
  \[T_s S_t x = T_s \int\limits_{[0,\infty)} T_rx \,\mu_t(\d r) = T_s \lim_{n\to\infty} \int\limits_{[0,n]} T_r x\,\mu_t(\d r) = \lim_{n\to \infty} T_s\int\limits_{[0,n]} T_r x\,\mu_t(\d r).\]
  Approximating $\int_{[0,n]} T_r x\,\mu_t(\d r)$ by (finite) Riemann sums, using the sequential continuity of $T_s$,  and taking into account the semigroup law for $(T_s)$, we conclude
  \[T_sS_t x = \lim_{n\to \infty} T_s\int\limits_{[0,n]} T_r x\,\mu_t(\d r) = \lim_{n\to \infty} \int\limits_{[0,n]} T_r T_s x\,\mu_t(\d r) = S_t T_s x\]
  again by Lemma \ref{lemma:integrability}.
  
  Let $\lambda>0$. By Lemma \ref{lem:properties_generator} we have
  \[
    \lambda (\lambda + A^f)^{-1} x = \lambda \int\limits_{(0, \infty)} \exp^{- \lambda t} S_t x \, \d t \quad(x\in X).
  \]
  Thus, the claims for $\lambda (\lambda + A^f)^{-1}$ follow from the claims for $(S_t)$ by approximating the integral by integrals over compact subsets and then by finite Riemann sums and taking into account the
  sequential closedness of $A$ and of the operators $T_t$ (which for them follows from sequential continuity).
\end{proof}

\begin{definition}
\label{def:f(A)}
  Let $X$ be a sequentially complete Hausdorff locally convex space, $(T_t)$ a (locally) sequentially equicontinuous, equibounded $C_0$-semigroup on $(X,\topo)$ with generator $-A$, and
  $f$ a Bernstein function with L\'{e}vy-Khintchine representation
  \[
    f(\lambda) = a + b \lambda + \int\limits_{(0,\infty)} \bigl(1-e^{-\lambda t}\bigr)\,\mu(\d t) \quad(\lambda>0). 
  \]
  We define the linear operator $A_f$ on $X$ by $\mathcal{D}(A_f) := \mathcal{D}(A)$ and
  \begin{align}
    A_f x & := ax + b Ax + \int\limits_{(0,\infty)} ( x - T_t x ) \mu(\d t) \quad(x\in \mathcal{D}(A_f))
    \label{eq:definition_A_f}
  \end{align}  
  where the integral exists by Proposition \ref{prop:pettis_integrability_of_Ttx-x}. 
\end{definition}
	
\begin{theorem}
\label{thm:representationonDA}
  Let $X$ be a sequentially complete Hausdorff locally convex space, $(T_t)$ a (locally) sequentially equicontinuous, equibounded $C_0$-semigroup on $(X,\topo)$ with generator $-A$, 
  $f$ a Bernstein function, and $-A^f$ the generator of the subordinated semigroup $(S_t)$ to $(T_t)$ w.r.t.\ $f$.
  Then $\mathcal{D}(A) \subseteq \mathcal{D}(A^f)$ and $A^f \big|_{\mathcal{D}(A)} = A_f$. 
\end{theorem}

\begin{remark}
  For Banach spaces $X$ this result is due to Phillips \cite{phillips1952}.
\end{remark}

\begin{proof}
  We adapt the proof of \cite[Theorem 13.6]{schilling2012} to our context. 
  Let $(a,b,\mu)$ be the L\'{e}vy triplet for $f$, and $(\mu_t)$ the associated family of measures.
  
  (i)
  Let us first assume that $f(0\rlim) = a = 0$, i.e.\ $(\mu_t)$ is actually a family of probability measures. 
  Then by Proposition \ref{prop:Laplace-Transform}
  \[
    f_n(\lambda) := \int\limits_{(0, \infty)} (1 - \exp^{-\lambda s}) \,n \mu_{\frac{1}{n}} (\d s) = \frac{1 - \exp^{-\frac{1}{n} f(\lambda)}}{\frac{1}{n}} \to f(\lambda) \text{ as } n \rightarrow \infty, 
  \]
  for all $\lambda>0$, i.e.\ $(f_n)$ is a sequence of Bernstein functions converging pointwise to $f$. 
  By \cite[Corollary 3.9]{schilling2012} we have
  \begin{align}
    \lim\limits_{n \rightarrow \infty} n \mu_{\frac{1}{n}} & = \mu \text{ vaguely on $(0,\infty)$}, \label{eq:vaguelimit}\\
    \lim\limits_{\substack{C \rightarrow \infty\\\mu(\{C\}) = 0}} \lim\limits_{n \rightarrow \infty} n \mu_{\frac{1}{n}} \bigl( [C, \infty) \bigr) & = 0,  \label{eq:upperlimit} \\
    \lim\limits_{\substack{c \rightarrow 0\rlim \\ \mu(\{c\}) = 0}} \lim\limits_{n \rightarrow \infty} \int\limits_{[0, c)} t \,n \mu_{\frac{1}{n}}(\d t) & = b. \label{eq:lowerlimit} 
  \end{align}
  Let $x\in\mathcal{D}(A)$ and $x'\in X'$.
  Let $c,C>0$ be such that $\mu(\{c,C\}) = 0$, i.e.\ $c$ and $C$ are points of continuity. 
  Then $n \mu_{\frac{1}{n}}|_{[c,C)} \to \mu|_{[c,C)}$ weakly since vague convergence implies
  \begin{align*}
   \mu\bigl([c,C)\bigr) & = \mu \bigl( (c,C) \bigr) \leq \liminf\limits_{n \to \infty} n \mu_{\frac{1}{n}} \bigl( [c,C) \bigr) \\
   & \leq \limsup\limits_{n \to \infty} n \mu_{\frac{1}{n}} \bigl( [c,C) \bigr) \leq \mu \bigl( [c,C] \bigr) = \mu\bigl([c,C)\bigr). 
  \end{align*}
  Now weak convergence follows from \cite[Theorem A.4]{schilling2012}.
  Hence, since the function $[[c,C)\ni t\mapsto \langle x', x - T_t x \rangle]$ is bounded and continuous, one has
  \[
    \lim\limits_{n \rightarrow \infty} \int\limits_{[c,C)} \langle x', x - T_t x \rangle \,n \mu_{\frac{1}{n}} (\d t ) = \int\limits_{[c,C)} \langle x' , x - T_t x \rangle \,\mu (\d t ).
  \]
  By dominated convergence, we obtain
  \[
    \lim\limits_{\substack{c \rightarrow 0\rlim \\ C \rightarrow \infty\\\mu(\{c,C\}) = 0}} \lim\limits_{n \rightarrow \infty} \int\limits_{[c,C)} \langle x' , x - T_t x \rangle \,n \mu_{\frac{1}{n}} (\d t ) = \int\limits_{(0,\infty)} \langle x' , x - T_t x \rangle \,\mu (\d t ).
  \]
  By Lemma \ref{lem:properties_generator}, for $c>0$ we have
  \begin{align*}
    \int\limits_{[0, c)} \langle x', x - T_t x \rangle \,n \mu_{\frac{1}{n}} ( \d t) = & \int\limits_{[0, c)} \int\limits_{(0,t)} \langle x' , T_s Ax - Ax \rangle \,\d s \, n \mu_{\frac{1}{n}} ( \d t)\\
  &  + \int\limits_{[0, c)} t \langle x' , Ax \rangle \,n \mu_{\frac{1}{n}} ( \d t). 
  \end{align*}
  Note that $[0,c)\ni t\mapsto \int_{(0,t)} \langle x' , T_s Ax - Ax \rangle \,\d s$ is continuous and bounded, and takes the value $0$ at $t=0$.
  Moreover, the sequence $\Bigl(\int_{(0,\cdot)} \langle x' , T_s Ax - Ax \rangle \,\d s \, n \mu_{\frac{1}{n}}\Bigr)_n$, interpreted as measures on $[0,c)$, is bounded and converges vaguely to the (finite) measure $\int_{(0,\cdot)} \langle x' , T_s Ax - Ax \rangle \,\d s \, \mu$ on $(0,c)$, which does not charge $\{0\}$.
  If $\mu(\{c\}) = 0$, we thus obtain that the convergence is even weakly. Hence, for such $c$ we have 
  \[
    \lim_{n \rightarrow \infty} \int\limits_{[0, c)} \int\limits_{(0,t)} \langle x' , T_s Ax - Ax \rangle \,\d s \, n \mu_{\frac{1}{n}} ( \d t) = \int\limits_{[0, c)} \int\limits_{(0,t)} \langle x' , T_s Ax - Ax \rangle \,\d s  \mu ( \d t),
  \]
  and therefore
  \[
    \lim\limits_{\substack{c \rightarrow 0\rlim \\ \mu(\{c\}) = 0}} \lim_{n \rightarrow \infty}\int\limits_{[0, c)} \int\limits_{(0,t)} \langle x' , T_s Ax - Ax \rangle \,\d s \, n \mu_{\frac{1}{n}} ( \d t)  = 0.
  \]
  Moreover,
  \[
    \lim\limits_{\substack{c \rightarrow 0\rlim \\ \mu(\{c\}) = 0}} \lim_{n \rightarrow \infty} \int\limits_{[0, c)} t \langle x' , Ax \rangle \,n \mu_{\frac{1}{n}} ( \d t) = b \, \langle x' , Ax \rangle
  \]
  by \eqref{eq:lowerlimit}.
  Hence,
  \[
     \lim\limits_{\substack{c \rightarrow 0\rlim \\ \mu(\{c\}) = 0}} \lim_{n \rightarrow \infty} \int\limits_{[0, c)} \langle x', x - T_t x \rangle \,n \mu_{\frac{1}{n}} ( \d t) = b \, \langle x' , Ax \rangle.
  \]
  Furthermore, since $(T_t)$ is equibounded, by \eqref{eq:upperlimit} we obtain
  \[
    \lim\limits_{\substack{C \rightarrow \infty\\\mu(\{C\}) = 0}} \lim_{n \rightarrow \infty} \int\limits_{[C, \infty)} \langle x' , x - T_s x \rangle n \mu_{\frac{1}{n}} ( \d s) = 0.
  \]
  Thus,
  \begin{align*}
    \langle x' , A_f x \rangle  
    & = \; \Bigl\langle x' , b Ax + \int\limits_{(0, \infty)} (x - T_t x ) \,\mu (\d t) \Bigr\rangle \\ 
    & = \lim\limits_{\substack{c \rightarrow 0\rlim \\ \mu(\{c\}) = 0}} \lim_{n \rightarrow \infty} \int\limits_{[0, c)} \langle x' , x - T_t x \rangle \,n \mu_{\frac{1}{n}} (\d t)\\
    & + \lim\limits_{\substack{c \rightarrow 0\rlim \\ C \rightarrow \infty\\\mu(\{c,C\}) = 0}} \lim\limits_{n \rightarrow \infty} \int\limits_{[c, C)} \langle x', x - T_t x \rangle \,n \mu_{\frac{1}{n}} (\d t) \\
    & + \lim\limits_{\substack{C \rightarrow \infty\\\mu(\{C\}) = 0}} \lim_{n \rightarrow \infty} \int\limits_{[C, \infty)} \langle x' ,  x - T_t x \rangle \,n \mu_{\frac{1}{n}} (\d t)  \\ 
    & = \lim\limits_{\substack{c \rightarrow 0\rlim \\ C \rightarrow \infty\\\mu(\{c,C\}) = 0}} \lim_{n \rightarrow \infty} \int\limits_{[0, \infty)} \langle x' , x - T_t x \rangle \,n \mu_{\frac{1}{n}} (\d t)  
    = \lim_{n \rightarrow \infty} \langle x' , n (x - S_{\frac{1}{n}}x) \rangle.
  \end{align*}
  For $\lambda>0$ let $x_\lambda:=\lambda (\lambda + A^f)^{-1}x$.
  We now apply Lemma \ref{lemma:invarianceofDA} multiple times.
  Then, on the one hand, $x_\lambda\in \mathcal{D}(A^f)\cap \mathcal{D}(A)$, and moreover (approximating the integral by integrals over compacta and then by finite Riemann sums again)
  \begin{align*}
    \langle x' , \lambda(\lambda+A^f)^{-1} A_f x \rangle  
    & = \Bigl\langle x' , b Ax_\lambda + \int\limits_{(0, \infty)} (x_\lambda - T_t x_\lambda ) \,\mu (\d t) \Bigr\rangle \\
    & = \lim_{n \rightarrow \infty} \langle x' , n (x_\lambda - S_{\frac{1}{n}}x_\lambda) \rangle 
    = \langle x' , A^f x_\lambda\rangle.
  \end{align*}
  Since this holds true for all $x'\in X'$, we obtain
  \[\lambda(\lambda+A^f)^{-1} A_f x = A^f x_\lambda.\]  
  By Lemma \ref{lem:properties_generator} we have $\lambda (\lambda + A^f)^{-1}\to I$ strongly as $\lambda\to \infty$.
  Since $A^f$ is sequentially closed by Lemma \ref{lem:properties_generator}, we thus obtain $x\in \mathcal{D}(A^f)$ and
  $A^fx = A_f x$.
  
  (ii)
  For the general case $f(0\rlim) = a \geq 0$ consider $h := f - a$. 
  Then $h$ is a Bernstein function with $h(0\rlim) = 0$. Let $(\nu_t)$ be the associated family of sub-probability measures. 
  Then $(\mu_t)$ given by $\mu_t = \exp^{-ta} \nu_t$ for $t \geq 0$ is the family of measures associated to $f$. 
  Thus, for $t\geq 0$ and $x\in X$ we have
  \[
    S^f_t x = \int\limits_{[0, \infty)} T_s x \, \mu_t (\d s) = \int\limits_{[0, \infty)} T_s x \, \exp^{-ta} \nu_t (\d s) = \exp^{-ta} S_t^h x,
  \]
  i.e.\ $(S_t^f)$ is a rescaling of $(S_t^h)$.
  Analogously to the case of $C_0$-semigroups on Banach spaces one proves that then $-A^f = -A^h - a$. Thus the general case follows from (i).
\end{proof}

\begin{remark}
  In the above proof, we first showed that $x\in \mathcal{D}(A)$ belongs to the weak generator of $(S_t)$, and then did a regularisation by resolvents to obtain the result.
  If $(S_t)$ is continuous (and not just sequentially continuous), we can directly conclude that $(x,Ax)$ belongs to the weak closure of $A^f$ which coincides with $A^f$ since $A^f$ is closed.
\end{remark}

\begin{corollary}
\label{coro:generator_subordinated_semigroup}
  Let $X$ be a sequentially complete Hausdorff locally convex space, $(T_t)$ a locally sequentially equicontinuous, equibounded $C_0$-semigroup on $(X,\topo)$ with generator $-A$, 
  $f$ a Bernstein function, and $-A^f$ the generator of the subordinated semigroup $(S_t)$ to $(T_t)$ w.r.t.\ $f$.
  Then $A_f$ is (sequentially) closable in $X$ and the sequential closure of $A_f$ equals $A^f$. 
\end{corollary}

\begin{proof}
  By Lemma \ref{lem:properties_generator} $\mathcal{D}(A)$ is sequentially dense in $X$, by Theorem \ref{thm:representationonDA} we have $\mathcal{D}(A)\subseteq \mathcal{D}(A^f)$ and
  by Lemma \ref{lemma:invarianceofDA} it is invariant under $(S_t)$.
  Thus, as in the case of $C_0$-semigroups on Banach spaces we conclude that $\mathcal{D}(A)$ is a `sequential core' for $A^f$, i.e.\ the sequential closure of $A^f|_{\mathcal{D}(A)}$ equals $A^f$ (see e.g.\ \cite[Proposition 1.14]{isem18} for the case of $C_0$-semigroups on Banach spaces).
  Since $A^f|_{\mathcal{D}(A)} = A_f$ by Theorem \ref{thm:representationonDA}, we obtain the assertion.
\end{proof}

Analogously to the scalar-valued case we shall write from now on
\[
 f(A) := A^f.  
\]
Note that similar to the situation in Banach spaces one could have developed an entire
functional calculus in the sense of \cite{haase2006} which enables one to define $f(A)$ with the
same outcome.

\section{Applications}
\label{chap:bi-continuous}

We now consider two applications, namely bi-continuous semigroups and transition semigroups of stochastic processes.

\subsection{Bi-continuous semigroups}

In this subsection let $X$ be a Banach space with norm-topology $\topo_{\norm{\cdot}}$.
 
\begin{definition}
  \label{defi:bicontinuoussg}
  Let $(T_t)$ in $L(X)$ be a semigroup on $X$ and $\topo$ a Hausdorff locally convex topology on $X$. We say that $(T_t)$ is \emph{(locally) bi-continuous (w.r.t.\ $\topo$)} if
  \begin{enumerate}
   \item
    $\topo\subseteq \topo_{\norm{\cdot}}$, $(X,\topo)$ is sequentially complete on $\norm{\cdot}$-bounded sets and $(X,\topo)' \subseteq (X,\topo_{\norm{\cdot}})'$ is norming for $X$,
   \item
     there exists $M\geq 1$, $\omega\in\R$ such that $\norm{T_t}\leq Me^{\omega t}$ for all $t\geq 0$,
    \item
      $(T_t)$ is a $C_0$-semigroup on $(X, \topo)$, 
    \item
     for every sequence $(x_n)$ in $X$, $x\in X$ with $\sup\limits_{n\in\N}\norm{x_n}<\infty$ and $\topo\text{-}\lim\limits_{n\to\infty} x_n = x$ we have
     \[\topo\text{-}\lim_{n\to\infty} T_t (x_n-x) = 0\]
     (locally) uniformly for $t\in [0,\infty)$.
  \end{enumerate}
  We say that a bi-continuous semigroup is \emph{uniformly bounded} if $\sup_{t\geq 0} \norm{T_t} <\infty$.
\end{definition}

\begin{remark}
      The notion of bi-continuous semigroups first appeared in \cite{Kuehnemund2001}, see also \cite[Definition 3]{kuehnemund2003}. 
\end{remark}

First, we study bi-continuous semigroups. In order to do this, we need some preparation.

\begin{remark} \label{remark:mixed_topo}
 \begin{enumerate}
    \item Given any Hausdorff locally convex topology $\topo$ coarser than $\topo_{\norm{\cdot}}$ one can construct a Hausdorff locally convex topology $\gamma := \gamma(\topo, \topo_{\norm{\cdot}})$, called \emph{mixed topology} 
      \cite[Section 2.1]{wiweger1961}, such that $\topo \subseteq \gamma \subseteq \topo_{\norm{\cdot}}$ and 
      $\gamma$ is the finest linear topology that coincides with $\topo$ on $\|\cdot\|$-bounded sets by 
      \cite[Lemmas 2.2.1, 2.2.2]{wiweger1961}.
    \item The triple $(X,\|\cdot\|,\topo)$ is called \emph{Saks space} by 
    \cite[Section I.3, 3.1 Lemma (c), 3.2 Definition]{cooper1978} 
    if $\topo$ is a Hausdorff locally convex topology on $X$ such that $\topo\subseteq \topo_{\norm{\cdot}}$ and 
    \begin{equation}\label{eq:saks}
    \norm{x}=\sup_{p\in\mathcal{P}} \norm{x}_p \quad (x\in X).
    \end{equation}
    Equation \eqref{eq:saks} is equivalent to the property that $(X,\topo)' \subseteq (X,\topo_{\norm{\cdot}})'$ 
    is norming for $X$ (cf.\ \cite[Remark 5.2]{budde2019} and \cite[Lemma 4.4]{kraiij2016}). 
    \item It is covenient to characterise the mixed topology $\gamma := \gamma(\topo, \topo_{\norm{\cdot}})$ by its 
    continuous seminorms. In the case that \eqref{eq:saks} holds a useful representation of these seminorms 
    may be given in the following way. For a sequence $(p_{n})$ in $\mathcal{P}$ and a 
    sequence $(a_{n})$ in $(0,\infty)$ with $\lim_{n\to\infty}a_{n}=\infty$ we define the seminorm
    \[
    \vertiii{x}_{(p_n),(a_n)}:=\sup_{n\in\N}\norm{x}_{p_n}a_{n}^{-1}\quad (x\in X).
    \]
    If either 
     \begin{enumerate}
     \item for every $x\in X$, $\varepsilon>0$ and $p\in\mathcal{P}$ there are $y,z\in X$ such that $x=y+z$, 
     $\norm{z}_p=0$ and $\norm{y}\leq \norm{x}_p+\varepsilon$, or 
     \item the $\norm{\cdot}$-unit ball $B_{\|\cdot\|}:=\{x\in X \mid \|x\|\leq 1\}$ is $\topo$-compact,
     \end{enumerate}
    then $\gamma$ is generated by $(\vertiii{\cdot}_{(p_n),(a_n)})$ due to \cite[Section I.4, 4.5 Proposition]{cooper1978}. 
 \end{enumerate}
\end{remark}

\begin{example}
\label{ex:strict}
	Let $X:=C_{\rm b}(\R^n)$ with supremum norm $\norm{\cdot}_\infty$, $\topo_{\rm co}$ the compact-open topology, 
	i.e.\ the topology of uniform convergence on compact subsets of $\R^n$, and $\gamma$ the mixed topology determined 
	by $\topo_{\norm{\cdot}_\infty}$ and $\topo_{\rm co}$. $(X,\norm{\cdot}_\infty,\topo_{\rm co})$ is a Saks space 
	which fulfils condition (i) of Remark \ref{remark:mixed_topo} (c) and $\gamma$ is also generated by the 
	weighted seminorms $\norm{f}_g := \sup_{x \in \R^n} \abs{g(x) f(x)}$ for $f \in C_{\rm b} (\R^n)$ with 
	weights $g \in C_0 (\R^n)$ by \cite[Section II.1, 1.11 Proposition]{cooper1978} 
	(cf.\ \cite[Theorem 2.3, 2.4]{sentilles1972}).
\end{example}

The topology generated by the seminorms $\norm{\cdot}_g$, $g \in C_0 (\R^n)$, on $C_{\rm b}(\R^n)$ was introduced under the name \emph{strict topology}, denoted by $\beta$, in 
\cite[Definition, p.\ 97]{buck1958} and the example shows that the strict topology is indeed a mixed topology (cf.\ \cite[Proposition 3]{cooper1971} and also Remark \ref{rem:strict_topology} (b), (c) below).

\begin{lemma}
\label{lem:mixed_topology}
  Let $X$ be a Banach space with norm-topology $\topo_{\norm{\cdot}}$, $\topo\subseteq \topo_{\norm{\cdot}}$ a Hausdorff locally convex topology on $X$ such that $(X, \topo)'$ is norming for $X$, and $\gamma := \gamma(\topo, \topo_{\norm{\cdot}})$ the mixed topology. Then the following holds. 
  \begin{enumerate}
    \item      
      $(X, \gamma)'$ is norming for $X$.
    \item
      $(X,\topo)$ is sequentially complete on $\norm{\cdot}$-bounded sets if and only if $(X,\gamma)$ is sequentially complete.
  \end{enumerate}
\end{lemma}

\begin{proof} \leavevmode
  \begin{enumerate}
    \item      
      This is clear since $\topo\subseteq \gamma$, so $(X,\topo)'\subseteq (X,\gamma)'$.
    \item
     First, we remark that the norming property guarantees that condition (d) of \cite[Theorem 2.3.1]{wiweger1961}) 
and \cite[Corollary 2.3.2]{wiweger1961} is fulfilled.
      By \cite[Corollary 2.3.2]{wiweger1961} $(x_n)$ is a Cauchy sequence in $(X,\gamma)$ if and only if $(x_n)$ is a Cauchy sequence in $(X,\topo)$ and $(x_n)$ is $\norm{\cdot}$-bounded,
      and \cite[Theorem 2.3.1]{wiweger1961} yields that $(x_n)$ is convergent in $(X,\gamma)$ if and only if $(x_n)$ is convergent in $(X,\topo)$ and $(x_n)$ is $\norm{\cdot}$-bounded.
      Thus, the assertion follows.
  \end{enumerate}
\end{proof}

Let us recall the concept of a sequential space (see \cite[Proposition 1.1]{franklin1965}, \cite[p.\ 53]{engelking1989}). 
A subset $A$ of a topological space $(X,\topo)$ is called 
\emph{sequentially closed} if for every sequence $(x_{n})$ in $A$ converging to a point $x\in X$ the point $x$ is already in $A$. 
A subset $U$ of $(X,\topo)$ is called \emph{sequentially open} if every sequence $(x_{n})$ in $X$ converging to a point $x\in U$ is eventually in $U$.
A topological space $(X,\topo)$ is called a \emph{sequential space} if each sequentially closed subset of $X$ is closed.
Equivalently, $(X,\topo)$ is a sequential space if and only if each sequentially open subset of $X$ is open.
In particular, all first countable spaces are sequential spaces \cite[Theorem 1.6.14]{engelking1989} 
as well as all Montel (DF)-spaces \cite[Theorem 4.6]{kakolsaxon2002} like the space of tempered distributions or the space of distributions with compact support with the strong dual topology. 
A topological vector space $(X,\topo)$ is called \emph{convex-sequential} or \emph{C-sequential} 
if every convex sequentially open subset of $X$ is open (see \cite[p.\ 273]{snipes1973}). Obviously, every 
sequential topological vector space is C-sequential. Further, every bornological topological vector space is 
C-sequential by \cite[Theorem 8]{snipes1973}. The bornological spaces $\mathcal{D}(\R^{n})$ 
of test functions on $\R^{n}$ with its inductive limit topology and 
$\mathcal{D}(\R^{n})'$ of distributions with its strong dual topology are examples of $C$-sequential spaces 
that are not sequential by \cite[Th\'eor\`{e}me 5]{shirai1959} and \cite[Proposition 1]{dudley1971}.
For our next proofs we need a classification of C-sequential Hausdorff locally 
convex spaces. Let $(X,\topo)$ be a Hausdorff locally convex space and $\mathcal{U}^{+}$ 
be the collection of all absolutely convex subsets $U\subseteq X$ which satisfy the condition that every sequence 
$(x_{n})$ in $X$ converging to $0$ is eventually in $U$.
Then $\mathcal{U}^{+}$ is a zero neighbourhood basis for a Hausdorff locally convex topology 
$\topo^{+}\subseteq\topo$ on $X$, which is the finest Hausdorff locally convex topology on $X$ with the same 
convergent sequences as $\topo$ by \cite[Proposition 1.1]{webb1968}. 

\begin{proposition}[{\cite[Theorem 7.4]{wilansky1981}}]\label{prop:c_sequential}
Let $(X,\topo)$ be a Hausdorff locally convex space. Then the following assertions are equivalent.
\begin{enumerate}
\item $X$ is C-sequential.
\item $\topo^{+}=\topo$.
\item For any Hausdorff locally convex space $Y$ a linear map $f: X\to Y$ is continuous if and only if 
it is sequentially continuous. 
\end{enumerate}
\end{proposition}

\begin{proposition}
\label{prop:mixed_topology_sequential}
Let $X$ be a Banach space with norm-topology $\topo_{\norm{\cdot}}$ and $\topo\subseteq \topo_{\norm{\cdot}}$ 
a Hausdorff locally convex topology on $X$ such that $\topo$ is metrisable on the $\norm{\cdot}$-unit ball 
$B_{\|\cdot\|}$, and $\gamma:=\gamma(\topo, \topo_{\norm{\cdot}})$ the mixed topology.
Then $(X, \gamma)$ is a C-sequential space and $\gamma^{+}=\gamma$.
\end{proposition}
\begin{proof}
Let $Y$ be any Hausdorff locally convex space and $f: (X,\gamma)\to Y$ any linear sequentially continuous map. 
It follows from \cite[Section I.1, Proposition 1.9]{cooper1978} that $f$ is even continuous. 
We conclude that $(X, \gamma)$ is a $C$-sequential space and $\gamma^{+}=\gamma$ by Proposition \ref{prop:c_sequential}.
\end{proof}

In particular, Proposition \ref{prop:mixed_topology_sequential} is applicable if $(X,\topo)$ is metrisable, 
and implies the $\beta^{+}=\beta$ part in \cite[Theorem 8.1]{kraiij2016}. 
Further, Proposition \ref{prop:mixed_topology_sequential} gives a sufficient condition for $\gamma^{+}=\gamma$ that is simple to check and relevant for the relation between bi-continuous semigroups and \emph{SCLE semigroups} 
\cite[Section 7]{kraiij2016}.
If $(X, \gamma)$ is even metrisable or equivalently first countable \cite[Proposition 1.1.11 (ii)]{zalinescu2002}, 
then we are in the uninteresting situation that $\gamma=\topo_{\norm{\cdot}}$ 
by \cite[Section I.1, Proposition 1.15]{cooper1978}. However, $X=C_{\rm b}(\R^n)$ with the mixed ($=$strict) topology 
from Example \ref{ex:strict} is a C-sequential space by Proposition \ref{prop:mixed_topology_sequential} 
which is not metrisable (not even bornological or barrelled) since the strict topology 
does not coincide with the norm-topology (cf.\ \cite[Section II.1, Proposition 1.2 5)]{cooper1978}).  

\begin{proposition}
	\label{prop:uniqueness_mixed}
	Let $X$ be a Banach space with norm-topology $\topo_{\norm{\cdot}}$ and $\topo\subseteq \topo_{\norm{\cdot}}$ a Hausdorff locally convex topology
	such that $(X, \topo)'$ is norming and $(X, \topo)$ a sequentially complete C-sequential space. 
	Then $\topo = \gamma(\topo, \topo_{\norm{\cdot}})$, i.e.\ further mixing does not extend the topology.     
	\label{prop:mixing}  	
\end{proposition}

\begin{proof}
	Denote $\gamma:= \gamma(\topo, \topo_{\norm{\cdot}})$. 
	From \cite[Lemma 2.1.1 (3)]{wiweger1961} (condition (n) there is satisfied) one gets $\topo \subseteq \gamma$.  
	The other inclusion will be proved by contradiction. 
	Let us assume that there is $U\in\gamma$ such that $U\notin\topo$. Due to $(X,\topo)$ being C-sequential we 
	obtain $U\notin\topo^{+}$ by Proposition \ref{prop:c_sequential}.
	Since $U\notin\topo^{+}$, there is $x_{0}\in U$ such that $U$ is not a neighbourhood of $x_{0}$ 
	w.r.t.\ $\topo^{+}$, i.e.\ for all $V\in\mathcal{U}^{+}$ it holds that 
	$x_{0}+V\nsubseteq U$. W.l.o.g.\ $x_{0}=0$ because $\topo^{+}$ is a locally convex topology. 
	As $U\in\gamma$ and $x_{0}=0$, there is an absolutely convex zero neighbourhood $V_{0}\subseteq U$ w.r.t.\ $\gamma$. 
	Then there is a sequence $(x_{n})$ in $X$ such that $x_n \to 0$ w.r.t.\ $\topo$ but $(x_{n})$ 
	is not eventually in $V_{0}$ because otherwise $V_{0}\in\mathcal{U}^{+}$ with $x_{0}+V_{0}=V_{0}\subseteq U$.
	The sequence $(x_{n})$ cannot be $\norm{\cdot}$-bounded otherwise it would follow that $x_n \to 0$ w.r.t.\ $\gamma$ 
	by \cite[Theorem 2.3.1]{wiweger1961} (the norming property implies condition (d) of \cite[Theorem 2.3.1]{wiweger1961}) 
	and thus that $(x_{n})$ is eventually in $V_{0}$. Therefore, $(x_{n})$ is $\norm{\cdot}$-unbounded. 
	W.l.o.g.\ we assume 
	\[
		\forall n \in \N: \, \norm{x_n} > n
	\] 
	and that all $x_n$ are distinct from each other. 
	By the norming property for $n\in\N$ there exists a $\topo$-continuous functional $f_n$ in the unit sphere of the dual space such that $\abs{\langle f_n , x_n \rangle} > n$. 
	The sets $\{ \langle f_n, x \rangle \mid n \in \N\}$ are bounded by $\norm{x}$ for every $x \in X$.  
	Further, the set $K := \{ x_n \mid n \in \N\} \cup \{0\}$ is compact w.r.t.\ $\topo$ since the sequence $(x_n)$ is convergent to $0$. 
	The topology on $K$ is metrisable as image of the continuous map
	\[
		g: \bigl\{ \tfrac{1}{n} \mid n \in \N \bigr\} \cup \{0\} \ni x \mapsto g(x) := \begin{cases} x_n & \text{ if } x=\tfrac{1}{n}, \\ 0 & \text{ if } x=0, \end{cases}
	\] 
	where the domain is equipped with the standard metric (\cite[Chapter IX, § 2.10, Proposition 17]{bourbaki1966}). 
	Using \cite[Remark 4.1.b]{voigt1992} we know that $\overline{\operatorname{cx}} \, K$ is compact and convex. 
	By a variant of the uniform boundedness principle \cite[Theorem 2.9]{rudin1991} we obtain the boundedness of $\bigcup_{n\in\N} f_n(\overline{\operatorname{cx}} \, K)$. 
	In particular, $(\langle f_n,x_n\rangle)$ is bounded, contradicting $\abs{\langle f_n,x_n\rangle}\to \infty$.
\end{proof}

We remark that the condition of sequential completeness of $(X, \topo)$ in the preceding proposition can be weakened to the metric convex compactness property 
since we only need the compactness of $\overline{\operatorname{cx}} \, K$.

\begin{corollary}
	\label{coro:bounded}
	Let $X$ be a Banach space with norm-topology $\topo_{\norm{\cdot}}$, $\topo\subseteq \topo_{\norm{\cdot}}$ a Hausdorff locally convex topology
	such that $(X, \topo)'$ is norming and $(X, \topo)$ is a sequentially complete C-sequential space. 
	Then $B\subseteq X$ is bounded in $(X, \topo)$ if and only if $B$ is bounded in $(X, \topo_{\norm{\cdot}})$. 
\end{corollary}

\begin{proof}
    Since $\topo_{\norm{\cdot}}$ is finer than $\topo$, the $\topo_{\norm{\cdot}}$-bounded sets are $\topo$-bounded.
	The other inclusion follows from Proposition \ref{prop:mixing} and \cite[Proposition 2.4.1]{wiweger1961}. 
\end{proof}

The next proposition shows that bi-continuity of a semigroup is equivalent to being a $C_0$-semigroup with respect to the corresponding mixed topology. Thus, bi-continuous semigroups give rise to examples for subordination.

\begin{proposition}
\label{prop:from_bi-continuity_to_seq_continuity_and_back}
  Let $X$ be a Banach space with norm topology $\topo_{\norm{\cdot}}$, $\topo\subseteq \topo_{\norm{\cdot}}$ a Hausdorff locally convex topology, 
  $\gamma:=\gamma(\topo,\topo_{\norm{\cdot}})$ the mixed topology, and $(T_t)$ a semigroup in $L(X)$.
  \begin{enumerate}
    \item Let $(T_t)$ be a (locally) bi-continuous semigroup w.r.t.\ $\topo$.
	  Then $(X,\gamma)$ is sequentially complete, $(X,\gamma)'$ is norming for $X$ and $(T_t)$ is a (locally) sequentially equicontinuous $C_0$-semigroup on $(X,\gamma)$.
    \item Let $(X,\topo)$ be a sequentially complete C-sequential space, $(X,\topo)'$ norming for $X$ and $(T_t)$ a (locally) sequentially equicontinuous $C_0$-semigroup on $(X,\topo)$.
	  Then $\topo=\gamma$ and $(T_t)$ is a (locally) bi-continuous semigroup w.r.t.\ $\topo$.
  \end{enumerate}
\end{proposition}

\begin{proof} \leavevmode
  \begin{enumerate} 
    \item 
      By Lemma \ref{lem:mixed_topology}, $(X,\gamma)$ is sequentially complete and $(X,\gamma)'$ is norming for $X$.
      By \cite[Theorem 2.3.1]{wiweger1961} and properties (b) and (d) in the definition of (local) bi-continuity, $(T_t)$ is (locally) sequentially equicontinuous on $(X,\gamma)$.
      Let $x\in X$, $(t_n)$ in $[0,\infty)$ with $t_n\to 0$. Then by (b), $(T_{t_n}x)_n$ is bounded and $T_{t_n}x\to x$ in $(X,\topo)$. 
      By \cite[Theorem 2.3.1]{wiweger1961}, $T_{t_n}x\to x$ in $(X,\gamma)$. Thus, $(T_t)$ is a $C_0$-semigroup on $(X,\gamma)$.      
    \item 
      Due to Proposition \ref{prop:uniqueness_mixed} we have $\topo=\gamma$. Next, we only have to show (b) and (d) in the definition of (local) bi-continuity. 
      Clearly, (local) sequential equicontinuity implies (d), and (b) is a consequence of
      Corollary \ref{coro:bounded} and \cite[Proposition 3.6 (ii)]{federico2016}. 
  \end{enumerate}
\end{proof}

\begin{remark}
  Note that generators of bi-continuous semigroups are Hille--Yosida operators by \cite[Proposition 10]{kuehnemund2003} while densely defined Hille--Yosida operators are precisely the generators of $C_0$-semigroups on Banach spaces $(X, \tau_{\norm{\cdot}})$. 
  Moreover, in reflexive Banach spaces, Hille--Yosida operators are always densely defined.  
\end{remark}

\begin{lemma}
\label{lem:uniformly_bounded_equibounded}
	Let $X$ be a Banach space with norm-topology $\topo_{\norm{\cdot}}$, $\topo \subseteq \topo_{\norm{\cdot}}$ a Hausdorff locally convex topology on $X$ 
	such that $(X,\topo)' \subseteq (X,\topo_{\norm{\cdot}})'$ is norming for $X$ and $(X, \topo)$ is a sequentially complete C-sequential space. 
	Let $(T_t)$ be a locally bi-continuous semigroup. Then $(T_t)$ is uniformly bounded if and only if $(T_t)$ is equibounded on $(X,\topo)$.
\end{lemma}

\begin{proof}
	Let $B:=\{T_t x\in X \mid x\in X, \,  \norm{x}\leq 1,\, t\geq 0\}$.
	Then $(T_t)$ is uniformly bounded if and only if $B$ is bounded in $(X, \topo_{\norm{\cdot}})$.
	By Corollary \ref{coro:bounded}, this is equivalent to boundedness of $B$ in $(X,\topo)$, which in turn is equivalent to equiboundedness of $(T_t)$.
\end{proof}

We can now combine Proposition \ref{prop:subordinatedsg} and Corollary \ref{coro:generator_subordinated_semigroup} to easily obtain the following.

\begin{theorem}
\label{thm:bi-continuous}
  Let $X$ be a Banach space with norm-topology $\topo_{\norm{\cdot}}$, $\topo \subseteq \topo_{\norm{\cdot}}$ a Hausdorff locally convex topology on $X$
  such that $(X,\topo)' \subseteq (X,\topo_{\norm{\cdot}})'$ is norming for 
  $X$ and $(X, \gamma)$ is a C-sequential space where $\gamma:=\gamma(\topo,\topo_{\norm{\cdot}})$ is the mixed topology.
  Let $(T_t)$ be a uniformly bounded (locally) bi-continuous semigroup on $X$ w.r.t.\ $\topo$ with generator $-A$ and $f$ a Bernstein function. 
  Then the subordinated semigroup $(S_t)$ to $(T_t)$ w.r.t.\ $f$ is uniformly bounded and (locally) bi-continuous w.r.t.\ $\topo$ as well 
  and its generator $-f(A)$ is given by the sequential closure of $-A_f$.
\end{theorem}

\begin{proof}
  Let us first show that $(S_t)$ is uniformly bounded and (locally) bi-continuous.
  First, we apply Proposition \ref{prop:from_bi-continuity_to_seq_continuity_and_back} (a) 
  and obtain that $(T_{t})$ is a uniformly bounded (locally) sequentially equicontinuous $C_0$-semigroup on $(X,\gamma)$, 
  $(X,\gamma)$ is sequentially complete and $(X,\gamma)'$ norming for $X$. 
  Since $(X, \gamma)$ is a C-sequential space, an application of Proposition \ref{prop:from_bi-continuity_to_seq_continuity_and_back} (b) 
  yields that $(T_{t})$ is a uniformly bounded (locally) bi-continuous semigroup w.r.t.\ $\gamma$. 
  From Lemma \ref{lem:uniformly_bounded_equibounded} we derive that $(T_t)$ is equibounded on $(X,\gamma)$. 
  Proposition \ref{prop:subordinatedsg} yields that $(S_t)$ is a (locally) sequentially equicontinuous and equibounded $C_0$-semigroup on $(X,\gamma)$. 
  Another application of Proposition \ref{prop:from_bi-continuity_to_seq_continuity_and_back} (b) and then of Lemma \ref{lem:uniformly_bounded_equibounded} 
  provides that $(S_{t})$ is a uniformly bounded (locally) bi-continuous semigroup w.r.t.\ $\gamma$ and thus w.r.t.\ $\topo$ as well. 
  
  Let $-f(A)$ be the generator of $(S_t)$. By Corollary \ref{coro:generator_subordinated_semigroup} we have that the sequential closure of $A_f$ coincides with $f(A)$.  
\end{proof}

\begin{remark}
  If $(T_t)$ is a bi-continuous semigroup with generator $-A$, but maybe not uniformly bounded, one may be tempted to first rescale the semigroup, then apply Theorem \ref{thm:bi-continuous} and then rescale the subordinated semigroup again. If $f$ is a Bernstein function, then this procedure ends up with a generator being an extension of $-\bigl(f(A+\omega)-\omega\bigr)$, where $\omega$ is the rescaling parameter.
\end{remark}

Analytic semigroups (\cite{EngelNagel2000,lunardi1995}) provide a basic example for bi-continuous semigroups and the generators coincide. 

\begin{lemma}
	Let $(T_t)$ be an analytic or $C_0$-semigroup on a Banach space $(X, \topo_{\norm{\cdot}})$ with generator $-A$ which is at the same time bi-continuous w.r.t.\ the topology $\topo$ with generator $-\widetilde{A}$. 
	Then $A$ = $\widetilde{A}$. 
	\label{lemma:gen_anal_is_gen_bi}
\end{lemma}

\begin{proof}
	We write $\topo_{\norm{\cdot}}$-$\int$ and $\topo$-$\int$ respectively in order to indicate w.r.t.\ which to\-po\-lo\-gy we integrate. 
	For $x \in X$ and $\lambda\in \rho(-A)\cap \rho(-\widetilde{A})$ we have 
	\[
		 (\lambda + A)^{-1} x =  \topo_{\norm{\cdot}} \text{-} \;\; \mathclap{\int\limits_{(0, \infty)}} \;\; \exp^{-\lambda t} T_t x \, \d t = \topo\text{-} \;\; \mathclap{\int\limits_{(0, \infty)}} \;\; \exp^{-\lambda t} T_t x \, \d t = (\lambda + \widetilde{A})^{-1} x,
	\]
	where the last equality follows from Lemma \ref{lem:properties_generator}. Thus, $A = \widetilde{A}$.
\end{proof}

\begin{example}
\label{ex:heat_semigroup}
  Let $X:=C_{\rm b}(\R^n)$ with supremum norm $\norm{\cdot}_\infty$, and $\topo_{\rm co}$ the compact-open topology, 
  i.e.\ the topology of uniform convergence on compact subsets of $\R^n$.
  Let $k\from (0,\infty)\times\R^n\to \R$, 
  \[k(t,x):=k_t(x):=\frac{1}{(4\pi t)^{n/2}} e^{-\frac{\norm{x}^2}{4t}} \quad(t>0, x\in\R^n)\]
  be the Gau{\ss}-Weierstra{\ss} kernel. For $t\geq 0$ define $T_t\in L(X)$ by
  \[T_tf := \begin{cases}
		f &, t=0,\\
		k_t * f &, t>0.
	    \end{cases}\]
  Then $(T_t)$ is a uniformly bounded analytic semigroup with generator $-A=\Delta$ on the domain 
  $\mathcal{D}(A) = \{ f \in C_{\rm b}(\R^n) \mid \forall p \geq 1: \, f \in W^{2,p}_{\rm loc} (\R^n) , \, \Delta f \in C_{\rm b} (\R^n) \}$ 
  (for the case $n\geq 2$; in case $n=1$ we have $\mathcal{D}(A) = C^2_{\rm b}(\R)$) in $C_{\rm b}(\R^n)$; cf.\ \cite[Propositions 2.3.1, 4.1.10]{BertoldiLorenzi2007}.
  It is locally bi-continuous w.r.t.\ $\topo_{\rm co}$ (\cite[Example 1.6]{Kuehnemund2001}) and by Lemma \ref{lemma:gen_anal_is_gen_bi} its generator is the operator $-A$ as introduced above. 
  Let $\alpha\in (0,1)$ and $f\from (0,\infty)\to [0,\infty)$ be defined by $f(x):=x^\alpha$ for all $x>0$.
  Let $(S_t)$ be the subordinated semigroup of $(T_t)$ w.r.t.\ $f$.
  Then $(S_t)$ is a uniformly bounded and locally bi-continuous semigroup, and the generator $-f(A)$ of $(S_t)$ is given by the sequential closure of $-(-\Delta)^{\alpha}$, i.e.\ the fractional Laplacian in $C_{\rm b}(\R^n)$, 
  as introduced in \cite[Sections 1.4, 5.6]{martinez2001}. 
\end{example}
\begin{proof}
We only need to prove the part on $(S_{t})$ which directly follows from Theorem \ref{thm:bi-continuous} 
since $(X,\gamma)$ with the mixed topology $\gamma:=\gamma(\topo_{\rm co},\topo_{\norm{\cdot}_\infty})$ is a C-sequential space by Proposition \ref{prop:mixed_topology_sequential}.
\end{proof}

Further examples of semigroups $(T_t)$ being strongly continuous for mixed topologies can be found e.g.\ in \cite{GoldysKocan2001}.

\begin{remark}
 The situation of Example \ref{ex:heat_semigroup} can be generalised. Namely, let $A$ be a sectorial operator such that $-A$ generates an analytic semigroup which is at the same time bi-continuous. 
 One can consider fractional powers $A^{\alpha}$, $\alpha \in (0,1)$, either by means of the standard sectorial functional calculus (see \cite[Chapter 3]{haase2006}) (for this one does not actually need that $-A$ generates a semigroup) or by using the methods from this paper. 
 Even without the assumption of dealing with a strongly continuous semigroup, one can still use \eqref{eq:definition_A_f}. 
 This still works since an analytic semigroup is always strongly continuous on $D := \overline{\mathcal{D}(A)}$. 
 It is essentially equivalent to using the Balakrishnan formula, see \cite[Proposition 3.1.12]{haase2006} for it and \cite[Proposition 3.2.1]{martinez2001} for the relation between the mentioned formulae. 
 Neither of the two approaches in general allows to obtain the `full' fractional power $A^{\alpha}$ defined by the sectorial calculus but only $(A_D)^{\alpha}$, the fractional power of $A_D$ which is the part of $A$ in the subspace $D$. 
 Unless $A$ is densely defined, $A_D$ is a proper restriction of $A$ which follows for example from \cite[Corollary 1.1.4 (iv)]{martinez2001}. 
 
 All those things are also true in rather general Hausdorff locally convex spaces, cf. \cite[Proposition 4.1.13, 4.1.22]{meichsner2021}. 
 
 For an operator $A$ as above this means that the here presented approach yields the same fractional powers as its sectorial functional calculus does in the Banach space $X$.  
\end{remark}

\subsection{Transition Semigroups for Markov Processes}

Let $(\Omega,\calF,\P)$ be a probability space and  $(E,\calE)$ a measurable space. 
Let us recall the notions of (normal) Markov processes and their associated transition semigroups.

\begin{definition}
  A tuple $\X:=(\Omega,\calF,\P,(\calF_t)_{t\geq 0},(X_t)_{t\geq 0}, E, \calE, (\P_x)_{x\in E})$ is called a \emph{Markov process} if
  $(X_t)$ is an adapted process on $(\Omega,\calF,\P)$ w.r.t.\ the filtration $(\calF_t)$ with values in $E$ and $(\P_x)$ is a family of probability measures on $(\Omega,\calF)$ such that
  $E\ni x\mapsto \P_x(X_t\in B)$ is measurable for all $B\in \calE$ and $\P_x(X_{t+s}\in B \mid \calF_s) = \P_{X_t}(X_s\in B)$ $\P_x$-a.s.\ for all $x\in E$, $t,s\geq 0$, and $B\in \calE$.
  A Markov process $\X$ is called \emph{normal} if $\{x\}\in \calE$ for all $x\in E$ and $\P_x(X_0=x)=1$ for all $x\in E$.
\end{definition}

We write $B_{\rm b}(E)$ for the bounded measurable (scalar) functions on $E$.

\begin{definition}
Let $\X:=(\Omega,\calF,\P,(\calF_t),(X_t), E, \calE, (\P_x)_{x\in E})$ be a Markov process.
For $t\geq 0$ we define $T_t\from B_{\rm b}(E)\to B_{\rm b}(E)$ by 
\[T_t f(x):= \E_x(f(X_t)) \quad(f\in B_{\rm b}(E), x\in E),\]
where $\E_x$ is the expectation with respect to $\P_x$.
We call $(T_t)$ the \emph{transition semigroup} associated with the Markov process.
\end{definition}

Transition semigroups $(T_t)$ for Markov processes satisfy a semigroup law, while normality of the Markov process yields that $T_0=I$. We state this well-known fact as a lemma.

\begin{lemma}
Let $(\Omega,\calF,\P,(\calF_t),(X_t), E, \calE, (\P_x)_{x\in E})$ be a normal Markov process with transition semigroup $(T_t)$. 
Then $(T_t)$ is a semigroup.
\end{lemma}

Now, let $E$ be a completely regular Hausdorff space, $\calE:=\calB(E)$ the Borel $\sigma$-field, and $\X$ a normal Markov process with transition semigroup $(T_t)$. Let us assume that $C_{\rm b}(E)$ is invariant for $(T_t)$, i.e. $T_t(C_{\rm b}(E))\subseteq C_{\rm b}(E)$ for all $t\geq 0$. Sometimes, $\X$ is then called a \emph{$C_{\rm b}$-Feller process} and $(T_t)$ a \emph{$C_{\rm b}$-Feller semigroup}, and we will adopt this notion. We may then try to restrict the transition semigroup $(T_t)$ to $C_{\rm b}(E)$.

We now introduce the strict topology on $C_{\rm b}(E)$ as in \cite{sentilles1972}.
\begin{definition}	
	Let $\beta E$ be the Stone--\v{C}ech compactification of $E$. For $Q\subseteq \beta E\setminus E$ compact we define
	\[C_Q(E):=\{g|_E\in C_{\rm b}(E) \mid\; g\in C(\beta E), g|_Q = 0\}.\]
	Then $C_Q(E)$ induces a topology $\beta_Q$ on $C_{\rm b}(E)$ via the seminorms $\norm{\cdot}_g$ for $g\in C_Q(E)$ given by $\norm{f}_g := \norm{gf}_\infty$.
	Then the \emph{strict topology} $\beta$ on $C_{\rm b}(E)$ is defined to be the inductive limit topology for $(C_Q(E), \beta_Q)_{Q\subseteq \beta E\setminus E \,\text{compact}}$.
\end{definition}

\begin{remark}
\label{rem:strict_topology}
    \begin{enumerate}
    \item 
    Equipped with the strict topology $\beta$, the space $(C_{\rm b}(E),\beta)$ is a Hausdorff locally convex space \cite[Theorem 2.1(b)]{sentilles1972}.
    \item
    If $E$ is $\sigma$-compact and locally compact, or Polish (i.e.\ complete metrisable and separable), then $\beta = \gamma(\topo_{\rm co},\topo_{\norm{\cdot}_\infty})$ and the strict topology is induced by the seminorms from Remark \ref{remark:mixed_topo} (c), by \cite[Theorem 2.4, Theorem 9.1]{sentilles1972}.
    \item
    If $E$ is $\sigma$-compact and locally compact, then the strict topology $\beta$ on $C_{\rm b}(E)$ is induced by the seminorms
    $\norm{\cdot}_g$ for $g\in C_0(E)$ given by $\norm{f}_g := \norm{gf}_\infty$ (see Example \ref{ex:strict}).
    \end{enumerate}
\end{remark}

Let us collect some results on $C_{\rm b}$-Feller semigroups on $(C_{\rm b}(E), \beta)$. To this end, we write $L_0((\Omega,\calF,\P);E)$ for the space of $\P$-equivalence classes of strongly measurable functions from $\Omega$ to $E$ equipped with the topology of convergence in measure.

Now, let $E$ be a complete metric space.

\begin{proposition}
\label{prop:transition_sg}
    Let $\X:=(\Omega,\calF,\P,(\calF_t),(X_t), E, \calE, (\P_x)_{x\in E})$ be a $C_{\rm b}$-Feller process and $(T_t)$ the associated $C_{\rm b}$-Feller semigroup on $(C_{\rm b}(E),\beta)$.
    \begin{enumerate}
        \item $(T_t)$ is a $C_0$-semigroup on $(C_{\rm b}(E),\beta)$ if and only if for all $f\in C_{\rm b}(E)$ we have $T_tf \to f$ uniformly on compact subsets of $E$.
        \item If $(T_t)$ is a $C_0$-semigroup on $(C_{\rm b}(E),\beta)$, then $(T_t)$ is locally equicontinuous, hence also locally sequentially equicontinuous.
        \item $(T_t)$ is equibounded.
        \item Let $X\from [0,\infty)\times E\to L_0((\Omega,\calF,\P);E)$, $X(t,x):=X_t$ where $X_0=x$, be continuous. Then $(T_t)$ is a $C_0$-semigroup on $(C_{\rm b}(E),\beta)$.
    \end{enumerate}
\end{proposition}

\begin{proof}
    (a) and (b) follow from \cite[Theorem 4.4]{kunze2009}.
    To show (c), first note that $(T_t)$ is contractive in $(C_{\rm b}(E),\tau_{\norm{\cdot}_\infty})$. Since $\beta$ and $\tau_{\norm{\cdot}_\infty}$ share the same bounded sets, $(T_t)$ is equibounded.
    (d) is a consequence of \cite[Theorem 5.2]{kunze2009}.
\end{proof}

Proposition \ref{prop:transition_sg} yields that as soon as the Markov process is continuous in time and initial value, then the corresponding transition semigroup satisfies all the properties needed for subordination.

\begin{theorem}
    Let $\X:=(\Omega,\calF,\P,(\calF_t),(X_t), E, \calE, (\P_x)_{x\in E})$ be a $C_{\rm b}$-Feller process and $(T_t)$ the associated $C_{\rm b}$-Feller semigroup on $(C_{\rm b}(E),\beta)$ with generator $-A$. Assume that $X\from [0,\infty)\times E\to L_0((\Omega,\calF,\P);E)$, $X(t,x):=X_t$ where $X_0=x$, is continuous. Let $f$ be a Bernstein function. Then $-f(A)$ is the generator of the subordinated semigroup.
\end{theorem}

\begin{proof}
    This is a direct consequence of Proposition \ref{prop:transition_sg} and Corollary \ref{coro:generator_subordinated_semigroup}.
\end{proof}

\bibliography{Subordination_BiStetigkeit}
\bibliographystyle{plainnat}
\end{document}